\documentclass[a4paper,12pt]{article}
\usepackage{amssymb,amsmath,graphicx,color}
\usepackage[colorlinks=true,citecolor=blue,urlcolor=blue,linkcolor=blue,bookmarksopen=true,unicode=true,pdffitwindow=true]{hyperref}
\usepackage{amsthm}
\usepackage[font=small,labelfont=bf,width=12.5cm]{caption}

\newtheorem{theorem}{Theorem}
\newtheorem{lemma}[theorem]{Lemma}

\newtheorem{proposition}[theorem]{Proposition}
\newtheorem{conjecture}[theorem]{Conjecture}

\newtheorem{problem}{Problem}

\newcommand{\num}{(3\,\Delta + 6)}

\newcommand{\set}[1]{\ensuremath{\left\{#1 \right\}}}

\title{Strong edge coloring of planar graphs}
\author{D{\'a}vid Hud{\'a}k\thanks{Institute of Mathematics, Faculty of Science, Pavol Jozef \v Saf\'arik University, Ko\v sice, Slovakia. 
	Supported in part by bilateral project SK-SI-0005-10 between Slovakia and Slovenia (all the authors), by VVGS 
	grant No. 617-B (I-10-032-00) (D. Hud\'ak), and by Slovak Research and Development Agency under contract 
	No. APVV-0023-10 and Slovak VEGA Grant No. 1/0652/12 (R. Sot\'ak). E-Mails: \texttt{david.hudak@student.upjs.sk, roman.sotak@upjs.sk}} 
\and Borut Lu{\v z}ar\thanks{Institute of Mathematics, Physics and Mechanics, Ljubljana, Slovenia. 
	Operation partially financed by the European Union, European Social Fund. E-Mail: \texttt{borut.luzar@gmail.com}} 
\and Roman Sot{\'a}k\footnotemark[1]  
\and Riste {\v S}krekovski\thanks{Department of Mathematics, University of Ljubljana, Ljubljana and Faculty of Information Studies, Novo mesto, Slovenia.  
	Partially supported by ARRS Program P1-0383. E-Mail: \texttt{skrekovski@gmail.com}}
}

\begin{document}
\maketitle
\abstract{\textit{A strong edge coloring} of a graph is a proper edge coloring where the edges at distance at most two receive
		  distinct colors. It is known that every planar graph with maximum degree $\Delta$ has a strong
		  edge coloring with at most $4\,\Delta + 4$ colors. We show that $3\,\Delta + 6$ colors suffice
		  if the graph has girth $6$, and $3\,\Delta$ colors suffice if the girth is at least $7$. Moreover,
		  we show that cubic planar graphs with girth at least $6$ can be strongly edge colored with at most $9$ colors.}

  
\section{Introduction}

A \textit{strong edge coloring} of a graph $G$ is a proper edge coloring where every color class induces a matching, i.e.,
every two edges at distance at most two receive distinct colors. The smallest number of colors for which 
a strong edge coloring of a graph $G$ exists is called the \textit{strong chromatic index}, $\chi_s'(G)$.
In 1985,  Erd\H{o}s and Ne\v{s}et\v{r}il posed the following conjecture during a seminar in Prague.
\begin{conjecture}[Erd\H{o}s, Ne\v{s}et\v{r}il]
	\label{con:basic}
	Let $G$ be a graph with maximum degree $\Delta$. Then,
	$$
		\chi_s'(G) \le \left \{ \begin{array}{cl}
								\frac{5}{4} \Delta^2\,, &\quad \Delta \textrm{ is even;} \vspace{0.3cm}\\ 
								\frac{1}{4}(5\Delta^2 - 2\Delta +1)\,, &\quad \Delta \textrm{ is odd.}
								\end{array} \right.
	$$
\end{conjecture}
\noindent They also presented the construction, which shows that Conjecture~\ref{con:basic}, if true, is tight.
In 1997, Molloy and Reed~\cite{MolRee97} established currently the best known upper bound for the strong chromatic index
of graphs with sufficiently large maximum degree.
\begin{theorem}[Molloy, Reed]
	For every graph $G$ with sufficiently large maximum degree $\Delta$ it holds that
	$$
		\chi_s'(G) \le 1.998\,\Delta^2.
	$$
\end{theorem}
In 1990, Faudree et al.~\cite{FauGyaSchTuz90} proposed several problems regarding subcubic graphs.
\begin{problem}[Faudree et al.]
	\label{prob:cubic}
	Let $G$ be a subcubic graph. Then,
	\begin{enumerate}
		\item{} $\chi_s'(G) \le 10$;
		\item{} $\chi_s'(G) \le 9$ if $G$ is bipartite;
		\item{} $\chi_s'(G) \le 9$ if $G$ is planar;
		\item{} $\chi_s'(G) \le 6$ if $G$ is bipartite and the weight of each edge is at most $5$;
		\item{} $\chi_s'(G) \le 7$ if $G$ is bipartite of girth $6$;
		\item{} $\chi_s'(G) \le 5$ if $G$ is bipartite and has girth large enough.
	\end{enumerate}
\end{problem}
Andersen~\cite{And92} confirmed that Conjecture~\ref{con:basic} holds for subcubic graphs, i.e., that the strong chromatic index of
any subcubic graph is at most $10$, which solves also the first item of Problem~\ref{prob:cubic}. The second item of Problem~\ref{prob:cubic}
was confirmed by Steger and Yu~\cite{SteYu93}.

In this paper, we consider planar graphs with bounded girth. In 1990, Faudree et al.~\cite{FauGyaSchTuz90} found a construction of planar
graphs showing that for every integer $k \ge 2$ there exists a planar graph $G$ with maximum degree $k$ and $\chi_s'(G) = 4\, k - 4$.
Moreover, they proved the following theorem.
\begin{theorem}[Faudree et al.]
	\label{thm:planar}
	Let $G$ be a planar graph with maximum degree $\Delta$. Then,
	$$
		\chi_s'(G) \le 4\, \Delta + 4 \,.
	$$
\end{theorem}
The proof of Theorem~\ref{thm:planar} is short and simple, so we present it here also.
\begin{proof}
	By Vizing's theorem~\cite{Viz64} every graph is $(\Delta + 1)$-edge colorable. Moreover, if $\Delta \ge 7$, $\Delta$ colors
	suffice~\cite{Zha00}. Let $M_i$ be the set of the edges colored by the same color. Let $G(M_i)$ be the graph obtained from $G$ where 
	every edge from $M_i$ is contracted. Note that the vertices corresponding to the edges of $M_i$ that are incident to a common edge 
	are adjacent in $G(M_i)$. Since $G(M_i)$ is planar, we can color the vertices with $4$ colors by the Four Color Theorem, and therefore
	all the edges of $M_i$ with a common edge receive distinct colors in $G$. After coloring each of the $k$ graphs $G(M_i)$, for
	$i \in \set{1,2,\dots,k}$ and $k$ being the chromatic index of $G$, we obtain a strong edge coloring of $G$.
\end{proof}

Notice that if a planar graph has girth at least $7$, the bound is decreased to $3 \chi'(G)$, due to Gr\"{o}tzsch's theorem~\cite{Gro59}.
Moreover, Kronk, Radlowski, and Franen~\cite{KroRadFra74} showed that if a planar graph has maximum degree $\Delta$ at least $4$ and girth at least $5$,
its chromatic index equals $\Delta$. This fact, combined with our result in Theorem~\ref{thm:subcubic} gives us the following.
\begin{theorem}
	Let $G$ be a planar graph with girth at least $7$ and maximum degree $\Delta$. Then,
	$$
		\chi_s'(G) \le 3\,\Delta\,.
	$$
\end{theorem}

Recently, Hocquard and Valicov~\cite{HocVal11} considered graphs with bounded maximum average degree. As a corollary they obtained
the following results regarding planar graphs.
\begin{theorem}[Hocquard, Valicov]
	\label{thm:hocval}
	Let $G$ be a planar subcubic graph with girth $g$. Then,
	\begin{itemize}
		\item[$(i)$] if $g \ge 30$, then $\chi_s'(G) \le 6$;
		\item[$(ii)$] if $g \ge 11$, then $\chi_s'(G) \le 7$;
		\item[$(iii)$] if $g \ge 9$, then $\chi_s'(G) \le 8$;
		\item[$(iv)$] if $g \ge 8$, then $\chi_s'(G) \le 9$.
	\end{itemize}
\end{theorem}
Later, these bounds were improved in~\cite{HocMonRasVal13} to the following:
\begin{theorem}[Hocquard et al.]
	\label{thm:hocval}
	Let $G$ be a planar subcubic graph with girth $g$. Then,
	\begin{itemize}
		\item[$(i)$] if $g \ge 14$, then $\chi_s'(G) \le 6$;
		\item[$(ii)$] if $g \ge 10$, then $\chi_s'(G) \le 7$;
		\item[$(iii)$] if $g \ge 8$, then $\chi_s'(G) \le 8$;
		\item[$(iv)$] if $g \ge 7$, then $\chi_s'(G) \le 9$.
	\end{itemize}
\end{theorem}

In this paper we consider planar graphs with girth $6$ and introduce the following results.
\begin{theorem}
	\label{thm:girth6}
	Let $G$ be a planar graph with girth at least $6$ and maximum degree $\Delta \ge 4$. Then,
	$$
		\chi_s'(G) \le 3\,\Delta + 6\,.
	$$
\end{theorem}

\begin{theorem}
	\label{thm:subcubic}
	Let $G$ be a subcubic planar graph with girth at least $6$. Then,
	$$
		\chi_s'(G) \le 9\,.
	$$
\end{theorem}
Theorem~\ref{thm:subcubic} partially solves the third item of Problem~\ref{prob:cubic}. The proposed bound, if true, is realized by
the complement of $C_6$ (see Fig.~\ref{fig:cubplan}).
\begin{figure}[htb]
	$$
		\includegraphics{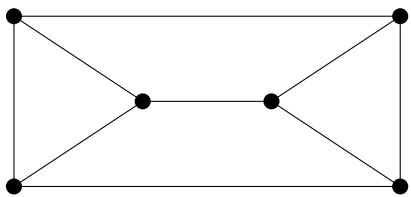}
	$$
	\caption{Subcubic planar graph with the strong chromatic index equal to $9$.}
	\label{fig:cubplan}
\end{figure}
Here, let us remark that very recently Hocquard et al.~\cite{HocMonRasVal13} obtained an improved result of Theorem~\ref{thm:subcubic},
proving that every subcubic planar graph without cycles of length $4$ and $5$ admits a strong edge coloring with at most $9$ colors.

\paragraph{}
All the graphs considered in the paper are simple. We say that a vertex of degree $k$, at least $k$, and at most $k$ is a \textit{$k$-vertex},
a \textit{$k^+$-vertex}, and a \textit{$k^-$-vertex}, respectively. Similarly, a \textit{$k$-neighbor}, a \textit{$k^+$-neighbor},
and a \textit{$k^-$-neighbor} of a vertex $v$ is a neighbor of $v$ of degree $k$, at least $k^+$, and at most $k$, respectively.
A \textit{$2$-neighborhood} of an edge $e$ is comprised of the edges at distance at most two from $e$. Here, the distance between the edges $e$ and $e'$
in a graph $G$ is defined as the distance between the vertices corresponding to $e$ and $e'$ in the line graph $L(G)$.

\section{Proof of Theorem~\ref{thm:girth6}}
\label{sec:girth}

We prove the theorem using the discharging method. First, we list some structural properties of a minimal counterexample $G$
to the theorem and then, we show that a planar graph with such properties cannot exist.

\subsection{Structure of minimal counterexample}

Since $G$ is a minimal counterexample, a graph obtained from
$G$ by removing any edge or vertex has some strong $\num$-edge coloring $\sigma$. In every proof we show
that $\sigma$ can be extended to $G$, establishing a contradiction. We note here that in the proofs
by removing an edge from a graph and adding it back after coloring, we do not decrease the distance
between the edges of the same color such that the distance between them would be less than $3$.

A $4$-vertex is a \textit{$4_2$-vertex} if it has at most two $2$-neighbors and a \textit{$4_3$-vertex} if it has three $2$-neighbors. 
We call a $2$-vertex \textit{weak} if it has a $3^-$-neighbor, 
\textit{semiweak} if it has a $4_3$-neighbor, and \textit{strong} otherwise.

First, we show that vertices of degree $1$ have neighbors of degree at least $5$.
\begin{lemma}
	\label{lm:onevertex}
	Every $1$-vertex in $G$ is adjacent to a $5^+$-vertex. Moreover, if it is adjacent to a $5$-vertex $u$,
	all the other neighbors of $u$ have degree at least $3$.
\end{lemma}

\begin{proof}
	Let $v$ be a $1$-vertex with a $4^-$-neighbor $u$. Let $\sigma$ be a strong edge coloring of $G - v$.
	The number of colored edges in the $2$-neighborhood of the edge $uv$ is at most $3\,\Delta$, hence there
	are at least $6$ colors available for $uv$, and so $\sigma$ can be extended to $G$, a contradiction.
	
	If $v$ is adjacent to a $5$-vertex $u$ with another $2^-$-neighbor, the number of colored edges in the 
	$2$-neighborhood of the edge $uv$ is at most $3\Delta + 2$, and so $\sigma$ can be extended again, since
	$uv$ has at least $4$ available colors.
\end{proof}

Next, in the minimimal counterexample $G$ every $2$-vertex is adjacent to some vertex of degree at least $5$.
\begin{lemma}
	\label{lm:twovertex}
	Every $2$-vertex has at least one $5^+$-neighbor.
\end{lemma}

\begin{proof}
	Suppose that $v$ is a $2$-vertex with two $4^-$-neighbors $u$ and $w$. Let $\sigma$ be the strong $\num$-edge
	coloring of $G - v$. Each of the two noncolored edges $uv$, $vw$ in $G$ has at most $3\,\Delta + 3$ colored 
	edges in the $2$-neighborhood, hence there are at least $3$ available colors for each, which means that both can
	be easily colored.
\end{proof}

The following lemma shows that every vertex of $G$ has at least one neighbor of degree at least $3$.
\begin{lemma}
	\label{lm:bigneighbor}
	Every vertex in $G$ has at least one $3^+$-neighbor.
\end{lemma}

\begin{proof}
	Suppose, to the contrary, that $v$ is a $k$-vertex of $G$ having only $2^-$-neighbors. Then, by the
	minimality of $G$, we have a strong $\num$-coloring $\sigma$ of $G - v$. By applying $\sigma$ to $G$, only the edges
	incident to $v$ remain noncolored. As each of these edges has at least $2\Delta - k + 7$ available colors,
	we can color them greedily one by one, and so extend $\sigma$ to $G$.
\end{proof}

Now, we prove that if a $k$-vertex of $G$ has $k-1$ $2^-$-neighbors, all of them are strong.
\begin{lemma}
	\label{lm:onlystrong}
	Every $k$-vertex, $k \ge 5$, with $k-1$ $2^-$-neighbors has only strong $2$-neighbors.
\end{lemma}

\begin{proof}
	Suppose that $v$ is a $k$-vertex, $k \ge 5$, with one $3^+$-neighbor and at least one $2$-neighbor $u$
	which is not strong and has another neighbor $w$. Since $u$ is not strong, $w$ is either a $4_3$-vertex or a $3^-$-vertex.
	Let $\sigma$ be a strong $\num$-edge coloring of $G - u$. The edge $uv$ has
	at most $\Delta + 2(k-2) + 3 \le 3\,\Delta - 1$ colored edges in the $2$-neighborhood, while the number of edges 
	in the $2$-neighborhood of $uw$ is at most $2\,\Delta + k \le 3\,\Delta$. Hence, both edges can be colored.
\end{proof}

Finally, we show that every $k$-vertex of $G$ with $k-2$ $2^-$-neighbors has at least three $2$-neighbors that are not weak.
\begin{lemma}
	\label{lm:weakstrong}
	Every $k$-vertex, $k \ge 5$, with $k-2$ $2^-$-neighbors has at least three nonweak $2$-neighbors.
\end{lemma}

\begin{proof}
	Let $v$ be a $k$-vertex, $k \ge 5$, with two $3^+$-neighbors. Moreover, suppose that $v$ is adjacent to 
	exactly two $2$-neighbors that are not weak. Let $u_1, u_2,\dots,u_{k-4}$ be the weak neighbors of $v$, and 
	$w_1,w_2,\dots,w_{k-4}$ their neighbors distinct from $v$. Let $\sigma$ be the strong $\num$-edge coloring
	of $G - \set{u_1,u_2,\dots,u_{k-4}}$. We extend $\sigma$ to $G$ in the following way. First, we color
	the edges $vu_1, vu_2,\dots, vu_{k-4}$ one by one. Such a coloring is possible since the number of
	colored edges in the $2$-neighborhood of such an edge never exceeds $3\,\Delta - 1$. Then, we color the 
	edges $u_iw_i$, $i \in \set{1,2,\dots, k-4}$. Again, every edge has at most $3\,\Delta$ colored edges 
	in the $2$-neighborhood, hence there is a free color by which we color it. This establishes the lemma.
\end{proof}

\subsection{Discharging}

Now, we show that a minimal counterexample $G$ with the described properties does not exist. In order to prove
this, we set the charges to all vertices and faces in such a way that the sum of all charges is negative.
Then, we redistribute charges among the vertices and faces so that every single object has non-negative charge,
clearly obtaining a contradiction on the existence of $G$.

The initial charge of vertices and faces is set as follows:
\begin{eqnarray*}
	\textrm{ch}_0(v) &=& 2\,d(v) - 6,\quad v \in V(G);\\
	\textrm{ch}_0(f) &=& l(f) - 6,\quad f \in F(G).
\end{eqnarray*}
By Euler's formula, it is easy to compute that the sum of all charges is $-12$.
We redistribute the charge among the vertices and faces by the following discharging rules:
\begin{itemize}
	\item[\textbf{(R1)}] Every face sends $2$ to every incident $1$-vertex.
	
	\item[\textbf{(R2)}] Every $5^+$-vertex sends $2$ to every adjacent $1$-vertex.
	
	\item[\textbf{(R3)}] Every $5^+$-vertex sends $2$ to every adjacent weak $2$-vertex.
	
	\item[\textbf{(R4)}] Every $5^+$-vertex sends $\frac{4}{3}$ to every adjacent semiweak $2$-vertex.
	
	\item[\textbf{(R5)}] Every $5^+$-vertex sends $1$ to every adjacent strong $2$-vertex.
	
	\item[\textbf{(R6)}] Every $4_2$-vertex sends $1$ to each of the adjacent $2$-vertices.
	
	\item[\textbf{(R7)}] Every $4_3$-vertex sends $\frac{2}{3}$ to each of the three adjacent $2$-vertices.
\end{itemize}

\medskip
Now, we are ready to prove Theorem~\ref{thm:girth6}.
\begin{proof}
	Suppose, to the contrary, that $G$ is a minimal counterexample to the theorem. We use the structural properties of $G$
	to show that after applying the discharging rules the charge of all vertices and faces is nonnegative. 
	
	First, consider the faces of $G$. It is easy to see that since $G$ has girth at least $6$ the initial charge of every face is 
	nonnegative. Faces only send charge by the rule (R1), i.e., to every incident $1$-vertex they send $2$ of charge. Let the 
	\textit{base} of $f$ be a face with all incident $1$-vertices removed. Notice that, by the girth condition, the bases of all faces have length at least $6$
	and that every incident $1$-vertex increases the length of the base by $2$. So, the number of $1$-vertices incident to a face $f$
	is at most $\frac{1}{2}(l(f) - 6)$ and the final charge of $f$ is at least $l(f) - 6 - 2\cdot \tfrac{1}{2}(l(f) - 6) = 0$.
	
	Now, we consider the final charge of a vertex $v$ regarding its degree:
	\begin{itemize}
		 \item{} \textit{$v$ is a $1$-vertex.} By Lemma~\ref{lm:onevertex}, the unique neighbor $u$ of $v$ is of degree at least $5$.
		 		By (R1), $v$ receives $2$ of charge from its incident face and $2$ of charge
		 		from $u$ by (R2). Hence it receives $4$ in total and its final charge is $0$.
		
		\item{} \textit{$v$ is a $2$-vertex.} By Lemma~\ref{lm:twovertex}, $v$ has at least one $5^+$-neighbor $u$. In order
				to have nonnegative charge, $v$ needs to receive at least $2$ of charge. If $v$ is
				weak, it receives $2$ from $u$ by (R3). In case when $v$ is semiweak, it receives $\frac{4}{3}$ from $u$ by (R4) and 
				$\frac{2}{3}$ from the $4_3$-neighbor by (R7), again $2$ in total. 
				Otherwise $v$ is strong and it receives $1$ from each of the two neighbors by (R5) and (R6).
				
		\item{} \textit{$v$ is a $3$-vertex.} The initial charge of $v$ is $0$ and it neither sends nor receives any charge, hence its
				final charge is also $0$.
				
		\item{} \textit{$v$ is a $4$-vertex.} By Lemma~\ref{lm:onevertex}, $v$ has no $1$-neighbor and by Lemma~\ref{lm:bigneighbor}, 
				$v$ has at most three $2$-neighbors. In case when $v$ has precisely three $2$-neighbors it sends $\frac{2}{3}$ to each
				of them by (R7), which is $2$ in total, otherwise it may send $1$ to each $2$-neighbor by (R6). 
				The final charge of $v$ is thus at least $8 - 6 - 2 = 0$.
				
		\item{} \textit{$v$ is a $5$-vertex.} Suppose first that $v$ is adjacent to a $1$-vertex $u$. By Lemma~\ref{lm:onevertex}, 
				$u$ is the only $2^-$-neighbor of $v$, hence $v$ sends $2$ by (R2) and its final charge is $2$. Therefore, we may
				assume that $v$ has no $1$-neighbor. If $v$ has at most two $2$-neighbors it sends at most $4$ of charge in total 
				by the rules (R3)--(R5), and it retains nonnegative charge. If $v$ has three $2$-neighbors
				none of them is weak by Lemma~\ref{lm:weakstrong}, and so it sends at most $4$ by (R4) or (R5). 
				Finally, if $v$ has four $2$-neighbors, all of them are strong by Lemma~\ref{lm:onlystrong}, so $v$ sends $4$ by (R5). It follows
				that $v$ has nonnegative final charge.		
		
		\item{} \textit{$v$ is a $k$-vertex, $k \ge 6$.} Let $n_1$ and $n_2$ be the numbers of $1$-neighbors and $2$-neighbors of $v$.
				By Lemma~\ref{lm:bigneighbor}, we have that $n_1 + n_2 \le k-1$. If $n_1 + n_2 = k - 1$, by Lemma~\ref{lm:onlystrong},
				it follows that $v$ has only strong $2$-neighbors, and so $n_1 = 0$. Hence the final charge of $v$ is 
				$2k - 6 - (k - 1) = k - 5 \ge 0$. In case when $n_1 + n_2 = k-2$, by Lemma~\ref{lm:weakstrong}, we have that there are 
				at least three nonweak $2$-neighbors of $v$, so its final charge is at least 
				$2k - 6 - 2(k - 2 - 3) - 3 \cdot \frac{4}{3} = 0$. Finally, if $n_1 + n_2 \le k-3$, the final charge of $v$ is
				at least $2k - 6 - 2(k-3) = 0$.
	\end{itemize}
	We have shown that the final charge of every vertex and face in $G$ is nonnegative, and so is the sum of all charges. Hence, a
	minimal counterexample to Theorem~\ref{thm:girth6} does not exist.
\end{proof}

\section{Proof of Theorem~\ref{thm:subcubic}}

To prove the theorem we follow the same procedure as in Section~\ref{sec:girth}. First, we list some properties of a
minimal counterexample to the theorem. Recall that $G$ is subcubic.

\subsection{Structure of minimal counterexample}

The first lemma considers the minimum degree and the neighborhood of $2$-vertices.
\begin{lemma}
	\label{lm:cubver}
	For a minimal counterexample $G$ the following claims hold:
	\begin{itemize}
		\item[$(a)$] the minimum degree of $G$ is at least $2$;
		\item[$(b)$] every $2$-vertex has two $3$-neighbors;
		\item[$(c)$] every $3$-vertex has at least two $3$-neighbors.
	\end{itemize}
\end{lemma}

\begin{proof}
	\begin{itemize}
		\item[$(a)$] Suppose $v$ is a $1$-vertex in $G$ and let $u$ be its unique neighbor. By the minimality, there is a 
					strong $9$-edge coloring $\sigma$ of $G-v$. It is easy to see that there are at most $6$ colored
					edges in the $2$-neighborhood of $uv$, hence we easily extend $\sigma$ to $G$.
		\item[$(b)$] Suppose, to the contrary, that $u$ and $v$ are adjacent $2$-vertices in $G$. Let $w$ be the second neighbor of $v$.
					Let $G'= G - v$ and $\sigma$ a strong edge coloring of $G'$ which, by the minimality of $G$, uses at most $9$ colors. 
					An easy calculation shows that $vw$ has at least two, and $uv$ has at least three available colors. Hence, $\sigma$ can be extended to $G$.
		\item[$(c)$] Let $v$ be a $3$-vertex with $2$-neighbors $u$ and $w$, and let $z$ be the second neighbor of $w$. 
					By the minimality, the graph $G' = G - w$ has a strong edge coloring $\sigma$ with at most $9$ colors. 
					Notice that $wz$ has at most eight used colors in the $2$-neighborhood, so we can color it. 
					Finally, we color $vw$ which also has at most eight colors used in the $2$-neighborhood, and so establish the claim.					
	\end{itemize}
\end{proof}

In the next two lemmas we consider $6$ and $7$-faces incident to $2$-vertices.
\begin{lemma}
	\label{lm:six}
	There is no $6$-face incident to a $2$-vertex in $G$.
\end{lemma}

\begin{proof}
	Let $f$ be a $6$-face with an incident $2$-vertex and let the vertices of $f$ be labeled as in Fig.~\ref{fig:six}.
	\begin{figure}[htb]
		$$
			\includegraphics{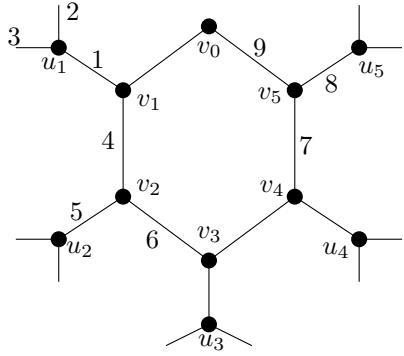}
		$$
		\caption{A $6$-face with an incident $2$-vertex and the nine distinct colors assigned to the edges
			in the $2$-neighborhood of $v_0v_1$.}
		\label{fig:six}
	\end{figure}
	By the minimality of $G$, there is a strong $9$-edge coloring $\sigma$ of $G' = G - v_0$. Now, we only need to color 
	the edges $v_0v_1$ and $v_0v_5$. Since there are only eight colored edges in the $2$-neighborhood
	of $v_0v_5$, we color it with a free color. 
	
	The edge $v_0v_1$ has now nine colored edges in the $2$-neighborhood. In 
	case that these nine edges only use at most eight colors, we can color $v_0v_1$ with a free color, thus we may assume that
	all nine edges have different colors assigned as it is shown in Fig.~\ref{fig:six}. Moreover, suppose that the color $1$ appears
	on $v_3v_4$, $v_4u_4$, or one of the edges incident to $u_5$. Then, there are at most seven distinct colors in the $2$-neighborhood
	of $v_0v_5$, so we may color it with a color different from $9$ and assign $9$ to $v_0v_1$. Similarly, the edge $v_4u_4$ and the
	edges incident to $u_5$ are not colored by $4$. Hence, we may assume that $\sigma(v_3v_4)$ is $2$ or $3$, say $2$, and the edge
	$v_4u_4$ and one of the edges incident to $u_5$ have the colors $3$ and $5$. The third edge incident to $u_5$ has color $6$ by
	the same argumentation.
		
	Next, one of the edges incident to $u_2$ has color $7$, otherwise we recolor $v_1v_2$ with $7$
	and color $v_0v_1$ with $4$. The same reasoning shows that the edge $v_3u_3$ or one of the edges incident to $u_2$ 
	has color $8$. Equivalently, one of the edges incident to $u_4$ has color $4$, for otherwise we recolor $v_4v_5$ with $4$
	and color $v_0v_1$ with $7$. Similarly, one of the edges $v_3u_3$ or an edge incident to $u_4$ has color $1$.
	
	It remains to consider three cases regarding the assignment of colors to the edges mentioned above.
	\begin{itemize}
		\item[$(a)$] \textit{Suppose that $\sigma(v_3u_3) = 8$.} Then, the third edge incident to $u_4$ has color
				$1$. Moreover, the third edge incident to $u_2$ has color $2$, otherwise we color $v_2v_3$
				with $2$ and $v_3v_4$ with $6$, which enables us to color $v_0v_1$ with $6$. 
				One of the edges incident to $u_3$ has color $4$, otherwise we color $v_2v_3$ with $4$,
				$v_1v_2$ with $6$, $v_0v_5$ with $4$, and finally $v_0v_1$ with $9$.				
				Similarly, one of the edges incident to $u_3$ has color $7$, otherwise we set $\sigma(v_3v_4) = 7$ and $\sigma(v_4v_5) = 2$ and 
				color $v_0v_1$ with $7$. But now, we can color $v_3v_4$ with $9$ and $v_0v_5$ with $2$ and set $\sigma(v_0v_1) = 9$, and so
				obtain the coloring of all the edges of $G$ (see Fig.~\ref{fig:six_1}).
				\begin{figure}[htb]
					$$
						\includegraphics{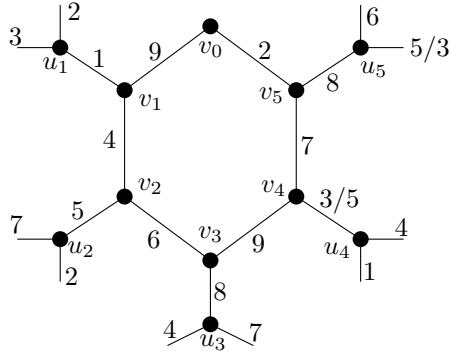}
					$$
					\caption{Coloring of the $6$-face in case $(a)$.}
					\label{fig:six_1}
				\end{figure}
		\item[$(b)$] \textit{Suppose that $\sigma(v_3u_3) = 1$.} Then, one of the edges incident to $u_2$ has color $8$. Therefore
				one of the edges incident to $u_4$ has color $6$, otherwise we set $\sigma(v_2v_3) = 2$, $\sigma(v_3v_4)=6$,
				and $\sigma(v_0v_1) = 6$. Similarly, the edges incident to $u_3$ must have colors $4$ and $7$. Otherwise, if there
				is no edge of color $4$, we swap the colors of $v_1v_2$ and $v_2v_3$, color $v_0v_1$ with $9$ and $v_0v_5$ with $4$.
				In case when there is no edge of color $7$ incident to $u_3$, we swap the colors of $v_3v_4$ and $v_4v_5$, and color $v_0v_1$ with $7$.								
				Finally, having such a coloring, we can swap the colors of $v_3v_4$ and $v_0v_5$ and color $v_0v_1$ with $9$ (see Fig.~\ref{fig:six_2}).
				\begin{figure}[htb]
					$$
						\includegraphics{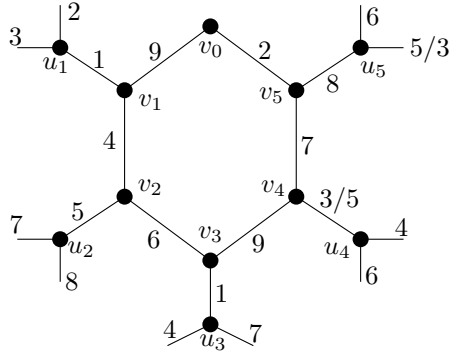}
					$$
					\caption{Coloring of the $6$-face in case $(b)$.}
					\label{fig:six_2}
				\end{figure}	
		\item[$(c)$] \textit{Suppose that one of the edges incident to $u_2$ has color $8$ and one of the edges
					incident to $u_4$ has color $1$.}	We can swap the colors of $v_2v_3$ and $v_3v_4$,
					which enables us to color $v_0v_1$ with $6$ and hence extend the coloring $\sigma$ to $G$.
	\end{itemize}
\end{proof}

\begin{lemma}
	\label{lm:seven}
	Every $7$-face in $G$ is incident to at most one $2$-vertex.
\end{lemma}

\begin{proof}
	By Lemma~\ref{lm:cubver}, we infer that a $7$-face with three incident $2$-vertices cannot appear in a minimal counterexample.
	Moreover, the arrangement of the possible two $2$-vertices incident to a $7$-face is unique. Let $f$ be such a face and let its 
	vertices be labeled as in Fig.~\ref{fig:7face}.
	\begin{figure}[htb]
		$$
			\includegraphics{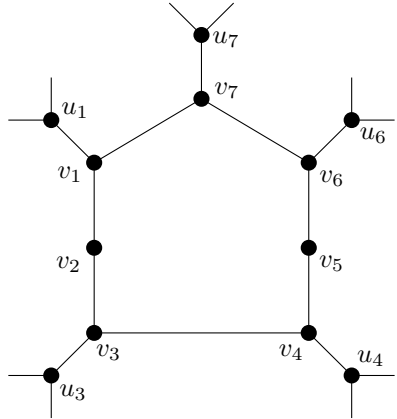}
		$$
		\caption{A $7$-face with two incident $2$-vertices.}
		\label{fig:7face}
	\end{figure}
	
	By the minimality of $G$, there exists a strong $9$-edge coloring $\sigma$ of $G - \set{v_2,v_3,v_4,v_5}$. 
	It is easy to see that after assigning the colors of $\sigma$ to the edges of $G$, there are seven
	noncolored edges. A simple count shows that the edge $v_1v_2$ has at most $6$ colored edges in the $2$-neighborhood and hence
	at least $3$ free colors. Similarly, the edges $v_2v_3$, $v_3v_4$, and $v_4v_5$ have at least $5$ free colors,
	and $v_3u_3$, $v_4u_4$, and $v_5v_6$ have at least $3$ free colors. 
	
	Now, color $v_4u_4$ with one of its free colors and $v_3v_4$ with the a color such that $v_1v_2$ retains at least	
	three available colors.
	Finally, color $v_3u_3$, $v_5v_6$, $v_4v_5$, $v_2v_3$, and $v_1v_2$ in the given order. It is easy to see that each of the 
	edges that are being colored always has at least one free color. Thus, the coloring $\sigma$ can be extended to $G$, a contradiction.
\end{proof}

From the proof of Lemma~\ref{lm:seven}, it also follows that the distance between the vertices of degree $2$ in $G$ is at least $4$.

\subsection{Discharging}

We use the same initial charge for vertices and faces as in the previous section, i.e.,
\begin{eqnarray*}
	\textrm{ch}_0(v) &=& 2\,d(v) - 6,\quad v \in V(G);\\
	\textrm{ch}_0(f) &=& l(f) - 6,\quad f \in F(G).\\
\end{eqnarray*}
We redistribute the charge using just one discharging rule:
\begin{itemize}
	\item[\textbf{(R)}] Every face sends $1$ to every incident $2$-vertex.
\end{itemize}

\medskip
Using the structure properties of a minimal counterexample $G$ and the discharging rule, we prove Theorem~\ref{thm:subcubic}.
\begin{proof}
	We show that after applying the discharging rule every vertex and face in $G$ has nonnegative charge which contradicts the
	fact that the total sum of initial charges is $-12$.
	
	By Lemma~\ref{lm:cubver}, there are only vertices of degree $2$ and $3$ in $G$. Every $2$-vertex is incident to two faces
	hence it receives $2$ of charge, so its final charge is $0$. On the other hand, $3$-vertices have initial charge $0$ and they
	send no charge, therefore their charge remains $0$.
	
	It remains to consider the faces. By Lemma~\ref{lm:six}, $6$-faces send no charge, so their charge remains $0$.
	By Lemma~\ref{lm:seven}, $7$-face sends at most $1$, hence it retains nonnegative charge. Finally, every $k$-face $f$,
	$k \ge 8$, has at most $\lfloor \frac{1}{3}l(f) \rfloor$ incident $2$-vertices by Lemma~\ref{lm:cubver}, therefore
	the final charge of $f$ is at least $l(f) - 6 - \lfloor \frac{1}{3}l(f) \rfloor \ge \lceil \frac{2}{3}l(f) \rceil - 6 \ge 0$.
\end{proof}

\section{Discussion}

In the introduction we mentioned that Faudree et al.~\cite{FauGyaSchTuz90} introduced a construction of planar graphs of girth $4$ with
the strong chromatic index equal to $4\Delta - 4$. We have shown that if the girth of a planar graph is at least $6$, $3 \,\Delta + 6$
colors suffice. This bound is not tight, however, there exist planar graphs with high girth and the strong chromatic index considerably close
to the bound given by Theorem~\ref{thm:girth6}. Consider an odd cycle of length $k$ and append $d - 2$, $d \ge 3$, leaves to each of the $k$ initial vertices 
(see Fig.~\ref{fig:conj} for an example). A planar graph with girth $k$ is obtained, we denote it by $C_k^d$. 
\begin{figure}[htb]
	$$
		\includegraphics{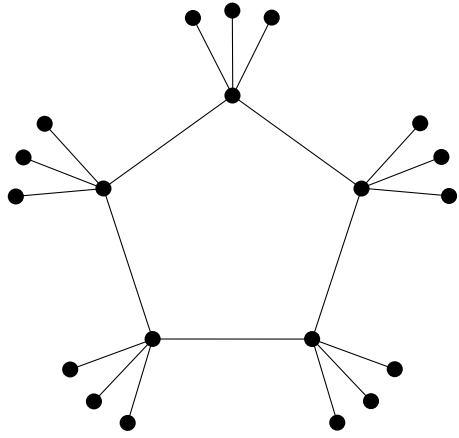}
	$$
	\caption{The strong chromatic index of $C_5^5$ is $10$.}
	\label{fig:conj}
\end{figure}	
\begin{proposition}
	\label{prop:frac}
	For every odd integer $k \ge 3$ and every integer $d \ge 3$, it holds
	$$
		\Bigg \lceil \frac{2k(d-1)}{k-1} \Bigg \rceil \le \chi_s'(C_k^d) \le \Bigg \lceil \frac{2k (d-2)}{k-1} \Bigg \rceil + 5\,.
	$$	
\end{proposition}
\begin{proof}
	First, consider the lower bound. There are $(d-1) k$ edges in $C_k^d$ and each color can be assigned to at most $(k-1) / 2$ of them in order 
	to satisfy the conditions of the strong edge coloring. Hence, the strong chromatic index of $C_k^d$ is at least $\lceil 2k \cdot (d-1) / (k-1) \rceil$. 
	
	To show that the upper bound holds, we first construct a strong edge coloring of the pendent edges with $2(d-2) + \ell$ colors, where
	$\ell = \big \lceil 2(d-2)/(k-1) \big \rceil$. 
	Let $v_1,v_2,\dots,v_k$ be the vertices of the cycle in $C_k^d$. The pendent edges incident to two consecutive vertices must receive distinct colors, 
	so $2(d-2)$ distinct colors must be used on the pendent edges incident to any pair of adjacent vertices of the cycle. We color the pendent edges 
	incident to the vertex $v_i$ by the colors in $C_i = \set{x_1^i,x_2^i,\dots,x_{d-2}^i}$, for $i \in \set{1,2,\dots,2 t + 1}$. Here, we define
	$t = \lceil (d-2) / \ell \rceil$, and $x_j^i = (j + (i-1)\cdot(d-2)) \mod{(2(d-2)+\ell)}$ for $j\in\set{1,2,\dots,d-2}$. 
	Observe that $2t + 1 \le k$.	%, e.g. by the following argumentation.
	%\begin{eqnarray*}
	%	t &\le& \frac{k-1}{2}	\\
	%	\Bigg \lceil \frac{d - 2}{\big \lceil \frac{2(d-2)}{k-1} \big \rceil} \Bigg \rceil &\le& \frac{k-1}{2} \\
	%\end{eqnarray*}
	%Recall that $k$ is odd, hence, the outer ceiling function may be omitted.
	%\begin{eqnarray*}
	%		\frac{d - 2}{\big \lceil \frac{2(d-2)}{k-1} \big \rceil} &\le& \frac{k-1}{2} \\
	%		2 (d-2) &\le& (k-1) \Bigg \lceil \frac{2(d-2)}{k-1} \Bigg \rceil \,,
	%\end{eqnarray*}
	%which obviously holds.
	
	Now, we color the pendent edges incident to the remaining vertices $v_{2t + 2}, v_{2t+3},\dots, v_k$ with the colors from the sets $C_{2t}$ 
	and $C_{2t+1}$. The pendent edges incident to the vertices with even indices receive the colors from $C_{2t}$, and the pendent edges incident to the
	vertices with odd indices receive the colors from $C_{2t + 1}$.
	Obviously, the pendent edges incident to two adjacent vertices $v_i$ and $v_{i+1}$, for $i \le k-1$, receive 
	distinct colors. Consider now the colors of the pendent edges incident to the vertex $v_{k}$. They receive the colors from $C_{2t+1}$, so we need to
	show that $C_{1} \cap C_{2t+1} = \emptyset$, i.e. $(d-2) \le 2t (d-2) \mod{(2(d-2)+\ell)} \le (d-2) + \ell$. Consider the following reduction:
	\begin{eqnarray*}
		d-2 \le 2\,t\,(d-2) \!\!\!\!&\mod{(2(d-2)+\ell)}& \le d-2 + \ell \\ 
		d-2 \le (2\,t\,(d-2) +t\,\ell - t\,\ell ) \!\!\!\!&\mod{(2(d-2)+\ell)}&  \le d-2 + \ell \\ 
		d-2 \le (2\,(d-2) +\ell - t\,\ell ) \!\!\!\!&\mod{(2(d-2)+\ell)}& \le d-2 + \ell \\ 
		0 \le ((d-2) +\ell - t\,\ell ) \!\!\!\!&\mod{(2(d-2)+\ell)}& \le \ell
	\end{eqnarray*}
	Obviously, $d-2 \le t\,\ell \le d-2 +\ell$, hence, $d-2 < x_j^k \le 2(d-2) + \ell$, for every color $x_j^k\in C_k$, so our coloring of the pendent 
	edges of $C_k^d$ is strong.	Finally, we use at most $5$ additional colors to color the edges of the cycle. This establishes the upper bound.
\end{proof}
Let us mention that $5$ additional colors are used only when $k = 5$, otherwise $4$ or $3$ colors suffice. Moreover, a longer argument shows that
instead of using only new colors on the cycle, some colors of the pendent edges could be used. Notice also that 
$\big \lceil 2k (d-2)/ (k-1) \big \rceil + 5 \le \big \lceil 2k (d-1)/(k-1) \big \rceil + 3$.

The graphs $C_k^d$ do not achive the highest strong chromatic index among the planar graphs of girth $k$ and maximum degree $d$ (see Fig.~\ref{fig:high} for an example).
\begin{figure}[ht]
	$$
		\includegraphics{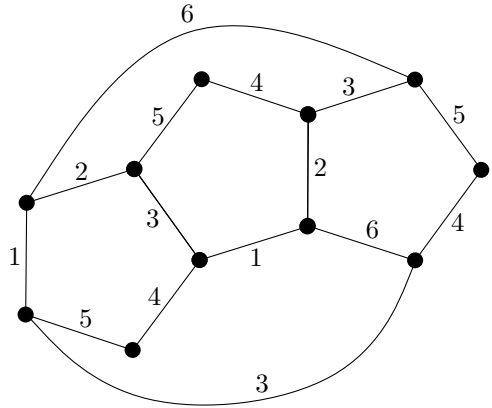}
	$$
	\caption{A planar graph with girth $5$, maximum degree $3$, the strong chromatic index $6$, and a strong edge coloring. 
		The graph $C_5^3$ has the strong chromatic index $5$.}
	\label{fig:high}
\end{figure}
However, we believe that they achieve it up to the constant, thus we propose the following conjecture.
\begin{conjecture}
	\label{conj:normal}
	There exists a constant $C$ such that for every planar graph $G$ of girth $k \ge 5$ and maximum degree $\Delta$ it holds
	$$
		\chi_s'(G) \le \Bigg \lceil \frac{2k (\Delta-1)}{k-1} \Bigg \rceil + C\,.
	$$
\end{conjecture}

%Moreover, it seems that the bound is related to fractional colorings. The \textit{fractional chromatic number}, $\chi_f(G)$, of a graph $G$ is the 
%minimum of $p/q$ such there exists a vertex coloring of $G$ with $p$ colors such that every vertex of $G$ is assigned $q$ colors and adjacent vertices 
%are assigned disjoint 
%sets of colors. By the classical result of the fractional graph theory (see e.g.~\cite{SchUll97}), $2n / (n-1)$  is the fractional chromatic number of an $n$-cycle.
%We propose the following.
%\begin{conjecture}
%	There exists a constant $C$ such that for every planar graph $G$ of maximum degree $\Delta$ it holds
%	$$
%		\chi_s'(G) \le \big \lceil (\Delta-1) \chi_f(G^*) \big \rceil + C\,,
%	$$	
%	where $G^*$ is the graph obtained from $G$ by removing all the vertices of degree $1$.
%\end{conjecture}

\bibliographystyle{acm}
{\small
	\bibliography{MainBase}

\begin{thebibliography}{10}

\bibitem{And92}
{\sc Andersen, L.~D.}
\newblock The strong chromatic index of a cubic graph is at most 10.
\newblock {\em Discrete Math. 108\/} (1992), 231--252.

\bibitem{FauGyaSchTuz90}
{\sc Faudree, R.~J., Gy\'{a}rf\'{a}s, A., Schelp, R.~H., and Tuza, Z.}
\newblock The strong chromatic index of graphs.
\newblock {\em Ars Combin. 29B\/} (1990), 205--211.

\bibitem{Gro59}
{\sc Gr\"{o}tzsch, H.}
\newblock Ein dreifarbensatz f\"{u}r dreikreisfreie netze auf der kugel.
\newblock {\em Wiss. Z. Martin-Luther-Univ. Halle-Wittenberg Math.-Natur. 8\/}
  (1959), 109--120.

\bibitem{HocMonRasVal13}
{\sc Hocquard, H., Montassier, M., Raspaud, A., and Valicov, P.}
\newblock On strong edge-colouring of subcubic graphs.
\newblock {\em Discrete Appl. Math.\/} (2013).
\newblock doi: http://dx.doi.org/10.1016/j.dam.2013.05.021.

\bibitem{HocVal11}
{\sc Hocquard, H., and Valicov, P.}
\newblock Strong edge colouring of subcubic graphs.
\newblock {\em Discrete Appl. Math. 159\/} (2011), 1650--1657.

\bibitem{KroRadFra74}
{\sc Kronk, H., Radlowski, M., and B.Franen}.
\newblock On the line chromatic number of triangle-free graphs.
\newblock {\em Abstract in Graph Theory Newsletter 3\/} (1974), 3.

\bibitem{MolRee97}
{\sc Molloy, M., and Reed, B.}
\newblock A bound on the strong chromatic index of a graph.
\newblock {\em J. Comb. Theory B 69\/} (1997), 103--109.

\bibitem{SteYu93}
{\sc Steger, A., and Yu, M.-L.}
\newblock On induced matchings.
\newblock {\em Discrete Math. 120\/} (1993), 291--295.

\bibitem{Viz64}
{\sc Vizing, V.~G.}
\newblock On an estimate of the chromatic class of a $p$-graph.
\newblock {\em Metody Diskret. Analiz 3\/} (1964), 25--30.

\bibitem{Zha00}
{\sc Zhang, L.}
\newblock Every planar graph with maximum degree 7 is of class 1.
\newblock {\em Graphs Combin. 16\/} (2000), 467--495.

\end{thebibliography}
}

\end{document}